\documentclass[11pt]{elsarticle}
\usepackage{geometry}               
\usepackage[parfill]{parskip}   
\usepackage{graphicx}
\usepackage{amssymb}
\usepackage{amsmath}
\usepackage{epstopdf}
\newtheorem{thm}{Theorem}
\newtheorem{lem}[thm]{Lemma}
\newtheorem{cor}[thm]{Corollary}
\newenvironment{proof}{\paragraph{Proof:}}{\hfill$\square$}
\title{Burning Circulant Graphs}
\author{Shannon Fitzpatrick \\
Leif Wilm\\
University of Prince Edward Island}
\begin{document}
\maketitle

\begin{abstract}
In  \cite{H2BG},  Bonato, Janssen and Roshanbin introduced the graph theoretic process of burning a graph.   This process begins with burning a single vertex of the graph, and at each subsequent time step two things occur: (1) the fire propagates and burns all neighbours of a previously burned vertex, and (2) a new vertex is selected to be burned.  The objective is to burn the entire graph in as few time steps as possible.

In this paper, we examine the burning numbers of circulant graphs.   The burning numbers are determined exactly  for 3-regular circulant graphs, and particular families of 4-regular circulant graphs.  We also provide upper and lower bounds on burning number for additional families of circulant graphs,  and provide some asymptotic results.

\end{abstract}

\section{Introduction}

With more and more people using social media, it has  increasingly become a source of news and information for many.  This is recognized by marketing firms who employ brand ``influencers" to promote their products on Facebook,  Instagram, Twitter, etc.  It is expected that the followers of these influencers will share, forward or re-tweet posts, and the original endorsement will propagate through their social network.   Given a number of influencers, all promoting the same product, but posting at different times, we ask ``How long would it take that post to propagate through a social network?"

In  \cite{H2BG},  Bonato, Janssen and Roshanbin introduced a simplified, deterministic version of this problem called {\it graph burning}.  Given a finite connected graph, the process of burning a graph begins with all vertices being unburned. At time step 1, a single vertex is chosen to be burned.  In each subsequent time step, two things occur: (1) the fire spreads to all neighbours of a previously burned vertex, and those vertices become burned, and (2) another vertex is selected to be burned.    Note that once a vertex is burned, it cannot be unburned.  The process is complete when all vertices of the graph have been burned.  The objective is to complete the process in the minimum number of time steps.   

Given a graph $G$,  suppose vertex $x_i$ is selected to be burned in time step $i$ of the burning process.  If all vertices of the graph are burned after  $k$ times steps, the sequence $(x_1, x_2, \ldots , x_k)$ is referred to as a {\it burning sequence}.   We note that once we choose to burn $x_i$ at time step $i$,  all vertices within distance $k-i$ of $x_i$ are burned by the end of time step $k$.  As a result,  $(x_1, x_2, \ldots , x_k)$ is a burning sequence of $G$ only if $N_0 [x_k]\cup N_1 [x_{k-1}]\cup \cdots \cup N_{k-1} [x_1]=V(G)$, where  $N_\ell [x] $ is the $\ell^{th}$-closed neighbourhood vertex $x$,   defined by $N_\ell [x] = \{y \in V(G): d(x,y) \le \ell\}$.

The minimum number of time steps required to burn all the vertices of a graph $G$ is called the {\it burning number of $G$} and denoted $b(G)$.    If a burning sequence of $G$ is of length $b(G)$, then we say it is an {\it optimal} burning sequence.   For example, $b(K_n) = 2$, when $n\ge 2$.  Furthermore, for any choice of distinct vertices $x_1$ and $x_2$ in $K_n$, $(x_1, x_2)$ is a burning sequence.   We note that after the choice of $x_1$, the fire propagates to all vertices of $K_n$ in time step 2, making the choice of $x_2$ redundant.  However, we still choose  vertex $x_2$ so that the length of the burning sequence is equal to the number of time steps used in the burning process.

The burning numbers of various classes of graphs have been determined; this  includes paths and cycles \cite{H2BG},  and complete bipartite graphs \cite{BGM}.  Upper and lower bound on the burning numbers of the Cartesian products and lexicographic products of graphs have also been determined \cite{BGM}.   In \cite{BGmodl}, it was shown that the burning number of a graph of order $n$ is bounded above by approximately $\sqrt{2n}$, and the authors conjectured that this could be improved to $\sqrt n$.   Progress has been made on this, with an upper bound of approximately $1.309 \sqrt{n}$ being established in  \cite{BBJ}.  

In this paper, we examine the burning number of circulant graphs.     Due to properties such as symmetry, scalability, and small average node distance, circulant graphs serve as good models for local area networks and parallel computer architectures \cite{Xu}.  (For a survey on circulant graphs see \cite{Hwang}.)  

The {\it circulant graph on $n$ vertices with distance set} $S$ has vertex set ${\mathbb Z}_n$  and edge set $ \{xy | x-y \in S\}$, where $S \subseteq { \mathbb Z}_n$ and $x \in S$ implies $-x \in S$, with addition  done modulo $n$.    We denote this circulant graph $C(n,S)$.   Circulant graphs are regular and vertex transitive, and are a subset of the more general family of Cayley graphs.  

Since $S$ can be written as $S= \{s_1, -s_1, s_2, -s_2, \ldots , s_t, -s_t\}$ where  $0 < s_1<s_2< \cdots < s_t \le n/2$,  we also use the notation $C(n;s_1, s_2, \ldots, s_t)$ to represent $C(n,S)$.   The notation $C(n;s_1, s_2, \ldots, s_t)$  is used in the majority of the paper, with the notation $C(n,S)$ mainly appearing in the final section regarding lexicographic products.    Also note that we limit our discussion to {\it connected} circulant graphs.  Therefore, we assume  $gcd(n,s_1,s_2, \ldots, s_t) =1$, since $S$ must generate $\mathbb Z_n$. 

In Section 2 of this paper, we begin by finding the burning number of $C(n;1, n/2)$ for any even integer $n$ such that $n \ge 4$.  In Section 3, we find upper and lower bounds for the class of $4$-regular circulant graphs $C(n;1,m)$ where $1< m < \frac n2$.  We also give improved bounds for the case when $n$ is a multiple of $m$, and show that, asymptotically, the burning number of $C(n;1,m)$ is $\sqrt \frac nm$ for a constant $m$ such that $1< m < \frac n2$.  In Section 4, we give the burning numbers for $C(n;1,m)$ for the specific cases $m = 2$ and $m=3$, improving the general results from Section 3 in these cases.     Finally, in Section 5, we examine the burning numbers of circulant graphs of higher degree, including those found by forming the lexicographic product of $C(n;1,m)$ with another circulant graph.

\section{Burning Numbers of 3-Regular Circulant Graphs}

We now consider the graph $C(n;1,m)$ where $n$ is even and $m  = n/2$.  This is the class of all 3-regular circulant graphs. In this section, we begin by finding the $\ell^{th}$ neighbourhood of each vertex.  
 
\begin{lem}\label{m2neigh}
Suppose $G \cong C(n;1,n/2)$, where $n$ is even and $n \ge 4$.  For any $x \in  V(G)$ and any $\ell$ such that $1 \le \ell \le n/4$,  $N_\ell [x] = [x-\ell, x+\ell] \cup [x - \ell +\frac n2+1, x + \ell +\frac n2 -1] $.
\end{lem}
\begin{proof}
Consider  vertex $i$ in $G$.  It can be easily verified that $N_1[i] = \{i -1, i, i+1, i+\frac n2\}$.  Therefore, $N_1[i] = [i-1,i+1] \cup [i+\frac n2, i+\frac n2]$.  We now proceed by induction.

Assume that $N_\ell [x] = [x-\ell, x+\ell] \cup \left [x+\frac n2 - \ell +1, x+\frac n2 + \ell -1 \right ] $, where $0 \le \ell \le \frac n4 -1$.  We now find  $N_{\ell+1}[0]$, noting that  $N_\ell [0] = [-\ell, \ell] \cup \left [\frac n2 - \ell +1, \frac n2 + \ell -1 \right ] $:
   \begin{eqnarray*}
   N_{\ell +1} [0] & =  & \bigcup_{i \in N_{\ell} [0]} N_1 [i]  \\
   &= & \bigcup_{i \in N_{\ell} [0]} \left \{ -1+i, i, i+1, i+\frac n2 \right\}  \\
   & = & [-\ell -1, \ell +1]  \cup\left [\frac n2 - \ell, \frac n2 + \ell \right]  \end{eqnarray*}

It can be similarly shown that $N_{\ell +1} [x] = [x-\ell-1, x+\ell+1] \cup\left [i+\frac n2 - \ell , i+\frac n2 + \ell \right] $.  The result follows by induction.  
\end{proof}

We now find the burning number of $C(n;1,n/2)$ for each $n \ge 4$.   A lower bound on $C(n;1,n/2)$ is first calculated using the cardinality of each of $\ell^{th}$ neighbourhoods.  The reader is then provided with a burning sequence that proves to be optimal.

\begin{thm}\label{half}
Suppose $G \cong C(n;1,n/2)$, where $n$ is even and $n \ge 4$.  Then $b(G) = \left \lceil \frac{1+\sqrt{2n+1}}{2} \right \rceil $.
\end{thm}

\begin{proof}
Let $k = b(G)$.  There is an optimal burning sequence $(x_1, x_2, \ldots, x_k)$, where $n = |N_0(x_k) \cup N_1(x_{k-1}) \cup \cdots \cup N_{k-1}(x_1)|$.  From Lemma \ref{m2neigh}, we find that $\left |N_\ell [x] \right | \le (2\ell + 1) + (2 \ell -1) = 4 \ell$ for any $x \in V(G)$ and $\ell \ge 1$.  We also note that $|N_0 (x)|= 1$.  Therefore, $\sum_{\ell = 0}^{k-1}|N_{i}[x_{k-i}]| \le1 + \sum_{\ell = 1}^{k-1} 4\ell =  1 + 2k^2 - 2k$, and  $n \le 
1 + 2k^2 - 2k$.  However, since $n$ is even, we have $n \le 2k^2 - 2k$.  It follows that $k \ge  \frac{1+\sqrt{2n+1}}{2}  $, and $b(G) \ge \left \lceil \frac{1+\sqrt{2n+1}}{2} \right \rceil $.

\medskip

Now let $k= \left \lceil \frac{1+\sqrt{2n+1}}{2} \right \rceil$, and consider the vertex sequence $(x_1, x_2, \ldots , x_k)$ where  
$$x_{j} = \begin{cases}
 \frac n2 -2kj-2k + j^2 -1  & :\,\mbox{j is even}\\
 -2kj-2k + j^2 -1 & :\,\mbox{j is odd and $j \neq k-1$}\\
\frac n2 +k & :\,\mbox{ j is odd and $j = k-1$}\\
 \end{cases}$$

Since $N_\ell [x] = [x-\ell, x+\ell] \cup [x+\frac n2 - \ell +1, x+\frac n2 + \ell -1] $, it follows that when $j$ is odd
\begin{eqnarray*}
N_{k-j}[x_j] &= &[-2kj+ j^2 +j +k-1, -2kj+j^2-j+3k-1] \\ & &  \cup \ \left [\frac n2 -2jk+j^2 + j + k, \frac n2 -2jk+j^2 - j + 3k-2 \right ]
\end{eqnarray*}
and when $j$ is even
\begin{eqnarray*}
N_{k-j}[x_{j}]  & = &\left [\frac n2-2kj+ j^2 +j +k-1, \frac n2-2kj+j^2-j+3k-1 \right ] \\ & &  \cup \ \left [-2jk+j^2 + j + k, -2jk+j^2 - j + 3k-2 \right ]
\end{eqnarray*}

\medskip

 For convenience, let $f(j) =  -2\,jk+{j}^{2}+j+k$ and $g(j) = \frac n2-2jk+j^2+j+k-1$.  
We claim that for any even $j$ such that $2 \le j \le k-1$, 
\begin{eqnarray*}
\bigcup_{i=1}^{j}N_{k-i}[x_i] & =  &  [ f(j), k-1] \cup \left [g(j), \frac n2 +k-2 \right ].  \end{eqnarray*}

%We claim that for any odd $j$,
%\begin{eqnarray*}
%\bigcup_{i=1}^{j+1}N_{k-i}[x_i] & =  &  [ -2\,jk+{j}^{2}+3j-k+2, k-1] \cup \left [\frac n2-2jk+j^2-k+3j+1, \frac n2 +k-2 \right ]   \end{eqnarray*}

%\begin{eqnarray*}
%N_{k-1}[x_1] \cup N_{k-2}[x_2] \cup \cdots \cup N_{k-j}[x_j] \cup N_{k-{(j+1)}}[x_{j+1}]   
%&  [ -2\,jk+3{j}^{2}+ \,j-k+2, k-1] \cup [-k+3j+\frac n2-2jk+j^2+1, \frac n2 +k-2]   \end{eqnarray*}

First, we verify the result for $j=2$: 
\begin{eqnarray*}
 N_{k-1}[x_1] \cup N_{k-2}[x_2]  & = & [-k+1, k-1] \cup \ \left [\frac n2 - k+2, \frac n2 +k-2 \right ] 
  \\  & &   \cup  \left  [ \frac n2-3k+5, \frac n2-k+1 \right ]  \cup \  [ -3k+6,-k] \\
  & = & [ -3k+6, k-1]  \cup  \left [ \frac n2-3k+5,\frac n2 +k-2 \right ]  \\
  & = &  [ f(2), k-1] \cup \left [g(2), \frac n2 +k-2 \right ]
  \end{eqnarray*}
  
Next, assume that the result holds for some even $j$ such that $1 \le j \le k-3$.  Hence, 

\begin{eqnarray*}
\bigcup_{i=1}^{j+2}N_{k-i}[x_i] & =  &  [f(j), k-1] \cup \left  [ g(j), \frac n2 +k-2 \right ]  \\
& & \cup \, N_{k-(j+1)}[x_{j+1}] \, \cup N_{k-(j+2)}[x_{j+2}] \\
    \end{eqnarray*}

 Since $j$ is even, we have
{\footnotesize { \begin{eqnarray*} 
 & & N_{k-(j+1)} [x_{j+1}] \, \cup \, N_{k-(j+2)}[x_{j+2}]  \\
 & = & [-2kj+j^2+3j-k+1, -2kj +j^2+j+k-1] \cup \left [ \frac n2 -2kj+j^2+3j-k+2, \frac n2 -2kj +j^2+j+k-2 \right ] \\
 &   \cup &\left [ \frac n2 -2kj+j^2+5j-3k+5, \frac n2 -2kj +j^2+3j-k+1 \right ] \, \cup \, [-2kj+j^2 +5j-3k+6, -2kj +j^2+3j-k] \\
& = & [-2kj+j^2 +5j-3k+6, -2kj +j^2+j+k-1] \cup  \left [ \frac n2 -2kj+j^2+5j-3k+5,\frac n2 -2kj +j^2+j+k-2 \right ] \\
& = & [f(j+2), f(j)-1] \cup  \left [ g(j+2),g(j) -1 \right ] \\
 \end{eqnarray*}}}
  Therefore,  
 $ \displaystyle \bigcup_{i=1}^{j+2}N_{k-i}[x_i]  =  [f(j+2), k-1] \cup  \left [g(j+2),\frac n2+k-2 \right ]$. We now consider two cases, based on the parity of $k$.

{\it Case 1: $k$ is even.}   By induction, when $k$ is even,  $k-2$ is even and $ \displaystyle \bigcup_{i=1}^{k-2}N_{k-i}[x_i]  =  [f(k-2), k-1] \cup  \left [g(k-2),\frac n2+k-2 \right ] = [-k^2+2k+2, k-1 ] \cup [\frac n2 -k^2+2k+1, \frac n2 +k-2]$.

Since  $k \ge \frac{1+\sqrt{2n+1}}{2}$, then $k(k-1) \ge \frac n2$.  Therefore, $-k^2 + 2k +2 \le - \frac n2 +k+2$ and $\frac n2 -k^2+2k+1 \le  k+1$.  Hence, $\bigcup_{i=1}^{k-2}N_{k-i}[x_i]  \supseteq [- \frac n2 +k+2, k-1 ] \cup [k+1, \frac n2 +k-2]$.

Now, if we have $x_{k-1} = \frac n2 + k$, then $N_1[x_{k-1} ]=\left  [\frac n2 +k - 1, \frac n2 +k+1 \right] \cup \{k\}$.
 Therefore, 
 $\bigcup_{i=1}^{k-1}N_{k-i}[x_i]  \supseteq [- \frac n2 +k+2,  \frac n2 +k+1] = [0,n-1]$.

\medskip

{\it Case 2: $k$ is odd.}   By induction, $ \displaystyle \bigcup_{i=1}^{k-1}N_{k-i}[x_i]  =  [f(k-1), k-1] \cup  \left [g(k-1),\frac n2+k-2 \right ] = [-k^2+2k, k-1] \cup [\frac n2 -k^2+2k-1, \frac n2 +k-2] $.
Therefore, {\footnotesize
\begin{eqnarray*}
   \bigcup_{i=1}^{k}N_{k-i}[x_i] 
& =  &[-k^2+2k, k-1] \cup\left [\frac n2 -k^2+2k-1, \frac n2 +k-2 \right ]  \cup N_0[x_k] \\
&=  &[-k^2+2k, k-1] \cup \left [\frac n2 -k^2+2k-1, \frac n2 +k-2\right ]  \cup \{-k^2-2k-1\} \\
&=  &[-k^2+2k-1, k-1] \cup \left [\frac n2 -k^2+2k-1, \frac n2 +k-2\right ]
\end{eqnarray*}}

We know $\frac{1+\sqrt{2n+1}}{2} \le k < \frac{1+\sqrt{2n+1}}{2}+1$.  Therefore, $-k^2 + k -1< -\frac{1+2\sqrt{2n+1} + 2n + 1}{4} + \frac{1+\sqrt{2n+1}}{2}  = - \frac n2 $.  It follows that   $-k^2 + 2k -1 \le  - \frac n2 + k -1$  and $\frac n2 -k^2 +2k-1 \le  k-1$.
Therefore,  $[ - \frac n2 + k-1, k-1] \subseteq [-k^2+2k-1, k-1]$ and $[k-1, \frac n2 +k-2] \subseteq  \left [\frac n2-k^2+2k-1, \frac n2 +k-2 \right ]$.  Therefore,  $   \bigcup_{i=1}^{k}N_{k-i}[x_i]  \supseteq [ - \frac n2 + k-1, k-1] \cup [k-1, \frac n2 +k-2] = \left [ - \frac n2 + k-1, \frac n2 +k-2\right ] = [0, n-1]$.   Therefore,  $(x_1, \ldots , x_k)$ is a burning sequence in $G$, and $b(G)\le \left \lceil\frac{1+\sqrt{2n+1}}{2} \right \rceil $.  The result follows.
\end{proof}

\section{Bounds on the Burning Numbers of 4-regular Circulant Graphs}

We now examine 4-regular circulant graphs of the form $C (n;1,m)$, where  $1 < m < n/2$.  Again, we begin by determining the $\ell^{th}$ neighbourhood of each vertex and its cardinality.
\begin{lem}\label{neighbourhoods}
Let $G \cong C(n;1,m)$ for some $m < n/2$.  For any $\ell \ge 0$, $N_\ell[x] = \bigcup_{j=0}^\ell ([x+ (j-\ell)m-j, x+(j-\ell)m+j] \cup [x+ (\ell-j)m-j, x+(\ell-j)m+j] )$.
\end{lem}

\begin{proof}
It is straightforward to verify that the result holds when $\ell = 0$ and $\ell = 1$.  We proceed by induction.  Assume that $N_k[x] = \bigcup_{j=0}^k ([x+ (j-k)m-j, x+(j-k)m+j] \cup [x+ (k-j)m-j, x+(k-j)m+j] )$ for some $k \ge 1$.  Now consider $N_{k +1} [0]= \bigcup_{i \in N_{k} [0]} N_1 [i]$.  

When comparing $N_{k}[0]$  and  $N_{k+1} [0]$, each interval  $[(j-k)m-j, (j-k)m+j]$ in the expression for $N_k[0]$ is replaced with  $[(j-k)m-j-1, (j-k)m+j+1]$ in $N_{k+1}[0]$.  Similarly, each $[ (k-j)m-j, (k-j)m+j] $ is replaced with $[ (k-j)m-j-1, (k-j)m+j+1] $ due to the appearance of 1 in the distance set.  Furthermore,  as a result of $m$ being in the distance set, only two additional vertices, $(-k - 1)m$ and $(k +1)m$  are added. Therefore, $N_{k+1}[0] = \bigcup_{j=0}^k ([ (j-k)m-j-1,(j-k)m+j+1] \cup [(k-j)m-j-1, (k-j)m+j+1]) \cup \{-(k+1)m, (k+1)m\} = \bigcup_{j=0}^{k+1} ([ (j-k)m-j,(j-k)m+j] \cup [(k-j)m-j, (k-j)m+j])$.

 By induction, the result follows for $x=0$.  By symmetry, the result follows for all $x$.
\end{proof}

\begin{cor}\label{neighboursize1}
Let $G \cong C(n;1,m)$ for some $m < n/2$.  For any $\ell \ge 0$ and vertex $x$ in $G$, $|N_\ell[x] | \le 2\ell^2 + 2 \ell +1$.
\end{cor}

\begin{proof}
Consider $N_\ell[0] = \bigcup_{j=0}^\ell ([ (j-\ell)m-j, (j-\ell)m+j] \cup [ (\ell-j)m-j, (\ell-j)m+j] )$ where $\ell \ge 0$. We note that $[ (j-\ell)m-j, (j-\ell)m+j] = [ (\ell-j)m-j, (\ell-j)m+j]  =[-\ell,\ell ]$ when $j = \ell$.  Furthermore, when $j \neq \ell$,  $|[ (j-\ell)m-j, (j-\ell)m+j]| = | [ (\ell-j)m-j, (\ell-j)m+j] | = 2j+1$.  It follows that $|N_\ell[x]| = |N_\ell[0]| \le \left ( \sum_{j=0}^{\ell-1} 2(2j+1) \right ) + (2\ell +1) =2\ell^2 + 2 \ell + 1$.
\end{proof}

We note that $|N_\ell[x] | = 2\ell^2 + 2 \ell +1$ is only achieved  when all the intervals in the union given in Lemma \ref{neighbourhoods} are disjoint.  When $\ell > \frac m2$, this will not be the case, and the upper bound on  $|N_\ell[x] |$ given in Corollary \ref{neighboursize1} can be improved.

\begin{lem}\label{neighbourhood1}
Let $G \cong C(n;1,m)$ where $1< m < n/2$.  For any $\ell$ such that $\ell > m/2$, 
{\small \begin{eqnarray*}
N_\ell [x] & =& \left ( \bigcup_{j=0}^{ \left \lfloor \frac m2 \right \rfloor -1} \left ( [x+ (j-\ell)m-j, x+(j-\ell)m+j] \cup [x+ (\ell-j)m-j, x+(\ell-j)m+j]  \right ) \right ) \\
& & \cup \left [ x- \left (\ell- \left \lfloor \frac m2 \right \rfloor \right )m - \left \lfloor \frac m2 \right \rfloor  , x+ \left (\ell- \left \lfloor \frac m2 \right \rfloor \right )m + \left \lfloor \frac m2 \right \rfloor \right ].
\end{eqnarray*}}
\end{lem}

\begin{proof}
Without loss of generality, consider $N_\ell [0]$ for some $\ell > \frac m2$.  From Lemma \ref{neighbourhoods}, it suffices to show that  $\bigcup_{j=\left \lfloor \frac m2 \right\rfloor}^\ell ([ (j-\ell)m-j, (j-\ell)m+j] \cup [(\ell-j)m-j, (\ell-j)m+j] ) = \left [ - \left (\ell- \left \lfloor \frac m2 \right \rfloor \right )m - \left \lfloor \frac m2 \right \rfloor , \left (\ell- \left \lfloor \frac m2 \right \rfloor \right )m + \left \lfloor \frac m2 \right \rfloor \right ]$ when $\ell \ge \left \lfloor \frac m2 \right \rfloor +1$.

{\it Claim: For $t$ such that $1 \le t \le \ell - \left \lfloor \frac m2\right \rfloor$,
$$\bigcup_{j=\ell -t }^\ell \left ([ (j-\ell)m-j, (j-\ell)m+j] \cup [(\ell-j)m-j, (\ell-j)m+j] \right) =  [ -tm -(\ell - t),tm +(\ell -t)].$$}

\medskip

{\it Proof of Claim:}
Since $\ell > \frac m2$, it follows that  $\ell \ge m - \ell +1$ and
 \begin{eqnarray*}
& &  \bigcup_{j=\ell -1 }^\ell ([(j-\ell)m-j, (j-\ell)m+j] \cup [ (\ell-j)m-j, (\ell-j)m+j] ) \\
& = & [ -m-\ell + 1,-m + \ell -1]   \cup [m - \ell +1,  m+ \ell -1 ]  \cup  [-\ell, \ell] \\
 & = & [ -m-\ell + 1,  m+ \ell -1 ] 
\end{eqnarray*}
 Hence, the claim is true when $t=1$.
 
Assuming that the claim holds for $t=w$, where $1\le w \le \ell -\left \lfloor \frac m2 \right \rfloor -1$, it follows that when $t=w +1$
 \begin{eqnarray*}
& & \bigcup_{j=\ell - w -1 }^\ell ([(j-\ell)m-j, (j-\ell)m+j] \cup [ (\ell-j)m-j, (\ell-j)m+j] ) \\
& = &   [-(w+1)m- \ell + w+), -(w+1)m+\ell - w -1] \\
& & \cup \, [ (w+1)m- \ell +w+1, (w+1)m+ \ell -w-1] \\
& &  \cup \, [ -wm -(\ell- w),wm +(\ell -w)] \\
&= &  [-(w+1)m- \ell + w+1,  (w+1)m+ \ell -w-1] 
\end{eqnarray*}
since $wm +(\ell -w) \ge  (w+1)m- \ell +w+1$.   Therefore, by induction, the claim holds for all $t$ such that $1 \le t \le \ell - \left \lfloor \frac m2 \right \rfloor$. 

\medskip

By letting $t = \ell - \left \lfloor \frac m2 \right \rfloor$ in the statement of the claim,  the result follows.
\end{proof}
\bigskip

\begin{cor}\label{neighboursize2}
Let $G \cong C(n;1,m)$ for some even $m$ such that $m < n/2$.  For any $\ell$ such that $\ell > m/2$, $|N_\ell [i]| \le 2 \left ( \left \lfloor \frac m2 \right \rfloor \right )^2 + 2\ell m -2  \left \lfloor \frac m2 \right \rfloor m +2  \left \lfloor \frac m2 \right \rfloor +1$.
\end{cor}

\begin{proof}  For any $\ell > m/2$, it follows immediately from Lemma \ref{neighbourhood1} that
\begin{eqnarray*}
 |N_\ell [i]| & \le  &\left (\sum_{j=0}^{\left \lfloor \frac m2  \right \rfloor -1} 2(2j+1) \right ) + 2\ell m -2  \left \lfloor \frac m2 \right \rfloor m +2  \left \lfloor \frac m2 \right \rfloor +1 \\
 &=  &2 \left ( \left \lfloor \frac m2 \right \rfloor \right )^2 + 2\ell m -2  \left \lfloor \frac m2 \right \rfloor m +2  \left \lfloor \frac m2 \right \rfloor +1 \\ 
 \end{eqnarray*}
 \end{proof}

Now that we have the cardinality of each of the $\ell^{th}$ neighbourhoods, we work to find lower bounds on the burning number of $C(n;1,m)$.  In the following result, we consider two bounds, the latter of which utilizes the improved lower bound on $|N_{\ell}[i]|$ when $\ell > m/2$ given in  Corollary \ref{neighboursize2} .
\begin{thm}\label{cnmLB}
For any circulant graph $C(n;1,m)$ such that $2 \le m < \frac n2$ 
$$b(C(n;1,m)) \ge \left \lceil  \frac{ \left ( 162n+6 \sqrt{729n^2+6} \right) ^{2/3}-6}{\left ( 162n+6 \sqrt{729n^2+6} \right )^{1/3}} \right \rceil .$$

Futhermore, when $m$ also satisfies $ \frac{1}{12}m^3 + \frac 12 m^2 + \frac 76 m +1 < n$, 

$$b(C(n;1,m)) \ge \begin{cases}
\left \lceil \frac{ 3m^2 +\sqrt{-3m^4+12m^2+144mn+36}-6 }{12m} \right \rceil & :\,\mbox{m is even}\\ \\
\left \lceil \frac{ 3m^2 +\sqrt{-3m^4-6m^2+144mn+9}-3}{12m} \right \rceil & :\, \mbox{m is odd}
\end{cases} .
$$

\end{thm}

\begin{proof}  Suppose $(x_1, x_2, \ldots , x_k)$ is an optimal burning sequence in $G$, where $G \cong C(n;1,m)$ such that $2 \le m < n/2$.
By Corollay  \ref{neighboursize1},   it follows that 
\begin{eqnarray*}
n &\le &\sum_{i=0}^{k-1}|N_i[x_{k-i}]| \\
& \le &\sum_{i=0}^{k-1}  (2i^2 + 2i +1)\\
&= & \frac 23 k^3 + \frac 13 k  .
\end{eqnarray*}

Solving this inequality for $k$, we obtain
$$k \ge \frac{ \left ( 162n+6 \sqrt{729n^2+6} \right) ^{2/3}-6}{\left ( 162n+6 \sqrt{729n^2+6} \right )^{1/3}}.$$
 
Now, assume that in addition to  $2 \le m  < \frac n2$, we also have $ \frac{1}{12}m^3 + \frac 12 m^2 + \frac 76 m +1 < n$.   It follows that $\frac m2 < k-1$.  This can be shown by assuming to the contrary that $k \le \frac m2 + 1$, and using  $n \le \frac 23 k^3 + \frac 13 k $.  

Since  $\frac m2 < k-1$, we have $\left \lfloor \frac{m}{2} \right \rfloor \le k-2$, and can use Corollaries  \ref{neighboursize1} and \ref{neighboursize2} to obtain an improved upper bound on $n$: 

\begin{eqnarray*} \label{maineq}
n &\le &\sum_{i=0}^{\left \lfloor \frac{m}{2} \right \rfloor} |N_i[x_{k-i}]| + \sum_{i=\left \lfloor \frac{m}{2} \right \rfloor +1}^{k-1}|N_i[x_{k-i}]| \\
 & \le &\sum_{i=0}^{\left \lfloor \frac{m}{2} \right \rfloor} (2i^2 + 2i +1)  +   \sum_{i=\left \lfloor \frac{m}{2} \right \rfloor +1}^{k-1} \left (2 \left ( \left \lfloor \frac m2 \right \rfloor \right )^2 + 2\ell m -2  \left \lfloor \frac m2 \right \rfloor m +2  \left \lfloor \frac m2 \right \rfloor +1 \right ) .
\end{eqnarray*}

We now consider two cases based on the parity of $m$.  

\medskip

 \noindent {\it Case 1:  $m$ is even.}   The previous inequality can be simplified as follows:
 
\begin{eqnarray*}
n & \le &\sum_{i=0}^{m/2} (2i^2 + 2i +1) + \sum_{i=m/2+1}^{k-1} ( 2i m - \frac{m^2}{2} + m +1)\\
& \le & mk^2 +\left  (1 - \frac{m^2}{2} \right )k + \frac{m^3}{12} - \frac{m}{3} . \\
\end{eqnarray*}

We now solve this inequality for $k$ and obtain the following: $$k \geq \frac{3m^2+\sqrt{-3m^4+12m^2+144mn+36}-6}{12m} . $$

\bigskip

\noindent
{\it Case 2: $m$ is odd.}  Again, we simplify the previous inequality relating $n$ and $k$ to obtain

\begin{eqnarray*}
n & \le &\sum_{i=0}^{ \frac{m-1}{2}} (2i^2 + 2i +1)  +   \sum_{i= \frac{m+1}{2}}^{k-1} \left (2 \left (  \frac {m-1}{2} \right )^2 + 2\ell m -2  \left ( \frac {m-1}{2} \right ) m +2 \left (  \frac {m-1}{2}  \right )+1 \right ) \\
& \le & mk^2+\left (\frac 12-\frac{m^2}{2} \right )k+\left ( \frac{m^3}{12}-\frac{m}{12}  \right ) .
\end{eqnarray*}

As in the previous case, we solve the inequality for $k$ and obtain the following:
 
$$k \geq \frac{3m^2+\sqrt{-3m^4-6m^2+144mn+9}-3}{12m} . $$

Note that since $n > \frac{1}{12} m^3$, each bound on $k$ is a real value.
\end{proof}

With lower bounds on the burning number of $C(n;1,m)$ established, we now turn our attention to  upper bounds.   For these graphs, we consider a class of induced subgraphs that are either paths or cycles and use the fact that the burning number of a path or cycle on $q$ vertices is $ \left \lceil \sqrt q \right \rceil$ \cite{H2BG}.
\begin{thm}\label{easierupper}
Let $G \cong C(n;1,m)$ such that $4 \le m \le n/2$.  Then $b(G)   < \sqrt{\frac nm} + \frac m2 +1 $.

\end{thm}

\begin{proof}
We know that $n = qm +r$ for some $q \ge 2$ and $0 \le r \le m-1$.  
For each $i =0, -1, 1, -2, 2,  \ldots  ,- \left \lfloor \frac m2 \right \rfloor,  \left \lfloor \frac m2 \right \rfloor$, let $H_i$ be the subgraph of $G$ induced on the vertex set $\{i, i+m, \ldots , i + (q-1)m\}$.  
Note that if $r=0$, each $H_i$ is a cycle on $q$ vertices.  Otherwise, each $H_i$ is a path on $q$ vertices.  

We now show that there is burning sequence of length at most  $ \left \lceil \sqrt q \right \rceil +  \left \lfloor \frac m2 \right \rfloor$.
Let $S = (x_1, x_2, \ldots , x_t)$ be an optimal burning sequence in $H_0$.   Hence,  $t = \left \lceil \sqrt q \right \rceil$.     Let $S'$ be a burning sequence of $G$ such that $S'$ has $x_1, x_2, \ldots , x_t$ as its first $t$ vertices.   We will assume that $S'$ has length at least $t +  \left \lfloor \frac m2 \right \rfloor$.

We note that for each vertex $x \in V(H_0)$, both $x+i$ and $x-i$ are in $N_i[x]$. It follows that, using the burning sequence $S'$,  all the vertices in $H_i$ and $H_{-i}$ are burned after $t+i$ time steps.  Therefore, if we let $P$ denote the set of vertices that are burned after $t +  \left \lfloor \frac m2 \right \rfloor$ time steps only as a result of the fire propagating from the initial subsequence $S$, then 
\begin{eqnarray*}
P &= &
\bigcup_{i= 1} ^{\left \lfloor \frac m2 \right \rfloor } \left ( V(H_i ) \cup  V(H_{-i}) \right ) \\
& =&  \bigcup_{j= 0} ^{q-1 } \left [ jm - \left \lfloor \frac m2 \right \rfloor,  jm + \left \lfloor \frac m2 \right \rfloor \right ]  \\
& = &\left [ -  \left \lfloor \frac m2 \right \rfloor,  (q-1)m +  \left \lfloor \frac m2 \right \rfloor \right ] 
\end{eqnarray*}

Assuming that $P \neq V(G)$, it follows that $P^c =  \left [qm  -  \left \lceil \frac m2 \right \rceil +1 ,  qm + r - \left \lfloor \frac m2 \right \rfloor -1 \right ] $, and $|P^c| =  r+  \left \lceil \frac m2 \right \rceil - \left \lfloor \frac m2 \right \rfloor -1 \le r$.    Therefore, after $t +  \left \lfloor \frac m2 \right \rfloor$ time steps at most $r$ vertices are not burned as a result of the propagation of the  subsequence $S$.

The vertices of  $P^c$ induce a path in $G$.  This path has an optimal burning sequence $(y_1, y_2, \ldots y_k)$ such that $k = \sqrt {|P^c|} \le \sqrt r$.  Since $r \le m-1$ and $m \ge 4$, it follows that $ \sqrt r  \le \left \lfloor \frac m2\right \rfloor$.   It follows that there is a burning sequence of length  at most $t +  \left \lfloor \frac m2 \right \rfloor$ in which the sequence  $(x_1, \ldots , x_t, y_1, \ldots y_k)$  serves as the first $t+k$ vertices in the burning sequence.  Furthermore, the fire propogating from these vertices burns all the vertices of $G$, and any additional vertices in the burning sequence are redundant.  It follows that $b(G) \le t +  \left \lfloor \frac m2 \right \rfloor$.

It follows that $b(G) \le \left \lceil \sqrt{q} \right \rceil + \left \lfloor \frac m2 \right \rfloor \le  \left \lceil \sqrt{ \frac nm } \  \right \rceil + \left \lfloor \frac m2 \right \rfloor < \sqrt{\frac nm} + \frac m2 +1 $.
\end{proof}

We now consider graphs $C(n;1,m)$ such that $n = mq$ for some $q\ge 3$.  We note that the subgraph $H$ induced on the vertices $\{0, m, \ldots, qm\}$ is a isometric cycle in $G$.  It was shown in \cite{H2BG} that under certain conditions, the burning number of an isometric subgraph of $H$ is at most the burning number of $G$.    In this result, given below, $N_\ell [x]$ refers to the $\ell^{th}$ neighbourhood of $x$ in $G$, while  $N_\ell ^H  [x]$ refers to the $\ell^{th}$ neighbourhood of $x$ in the subgraph $H$.

\begin{thm}\cite{H2BG} \label{isometric}
Let $H$ be an isometric subgraph of a graph $G$ such that, for any node $x \in V(G)\setminus V(H)$, and any positive integer $r$, there exists a node $f_r(x) \in V(H)$ for which $N_r[x] \cap V(H) \subseteq N^H_r [f_r(x)]$. Then we have that $b(H)\le b(G)$.
\end{thm}

\begin{cor}
For any $q >2$ and $m \ge 2$, $\left \lceil  \sqrt{q} \right \rceil \le b(C(mq; 1, m) )\le \left \lceil  \sqrt{q} \right \rceil + \left \lfloor \frac m2 \right \rfloor$.
\end{cor}

\begin{proof}
Let $G \cong b(C(mq; 1, m)$ for some $q \ge 3$ and $m \ge 2$.  The subgraph of $G$ induced on the set $\{0, m, 2m, \ldots , (q-1)m\}$ is a cycle of length $q$.  Furthermore, this cycle is an isometric subgraph of $G$. Call this subgraph $H$.  

Consider a vertex from the set $V(G) \setminus V(H)$.  Without loss of generality, we may assume that $0 < x \le \frac m2$.  From Lemma \ref{neighbourhoods}, we know that whenever $ \ell \ge 0$, 
 $N_\ell[x] = \bigcup_{j=0}^\ell ([x+ (j-\ell)m-j, x+(j-\ell)m+j] \cup [x+ (\ell-j)m-j, x+(\ell-j)m+j] )$.  We note that $x+ (j-\ell)m-j$ and $x+ (\ell-j)m-j$ are minimized when $j=0$ and $j=\ell$, respectively.  Similarly, $x+(j-\ell)m+j$ and $x+(\ell-j)m+j$ are maximized at $j=\ell$ and $j=0$, respectively.  It follows that $N_\ell[x] \subseteq [-\ell m + x, x+ \ell] \cup [x- \ell, x+\ell m] $.  Therefore, $N_\ell[x] \subseteq \{jm: -\ell +1 \le j \le \ell\}$.  Hence, $N_\ell[x] \subseteq N_\ell ^H [0]$.

By letting $f_r(x) = 0$ for each positive integer $r$, we satisfy the requirement in Theorem \ref{isometric}, and conclude that  $b(H) \le b(G)$.  Since $H$ is a cycle on $q$ vertices, $b(H) = \lceil \sqrt q \rceil$. This, together with Theorem \ref{easierupper}, gives the desired result. \end{proof}

\begin{thm}\label{order}
Taking $m$ to be constant, $b(C(n;1,m)$ is on the order of $\sqrt\frac nm$.
\end{thm}

\begin{proof}
Let $G = (n;1,m)$ where $n> 1+ \frac{m^3}{12}+\frac{m^2}{2} + \frac{7m}{6}$.  Let $$f(n)  =
 \begin{cases}
\left \lceil \frac{ 3m^2 +\sqrt{-3m^4+12m^2+144mn+36}-6 }{12m} \right \rceil & :\,\mbox{m is even}\\ \\
\left \lceil \frac{ 3m^2 +\sqrt{-3m^4-6m^2+144mn+9}-3}{12m} \right \rceil & :\, \mbox{m is odd}
\end{cases}
$$ and $g(n) = \sqrt{\frac nm} + \frac m2 +1 $.
By Theorems \ref{cnmLB} and \ref{easierupper}, $f(n)$  and $g(n)$ are  lower and upper bounds on $b(G)$, respectively.  It follows that $$  \lim_{n \rightarrow \infty} \frac{g(n)}{\sqrt\frac nm} \le  \lim_{n \rightarrow \infty}\frac{b(C(n;1,m)}{\sqrt\frac nm} \le  \lim_{n \rightarrow \infty} \frac{f(n)}{\sqrt\frac nm} $$ and $$ 1 \le  \lim_{n \rightarrow \infty}\frac{b(G)}{\sqrt\frac nm} \le1.$$
Therefore, $b(C(n;1,m)$ and $\sqrt{\frac nm}$ are asymptotically equal, and the result follows.
\end{proof}

\section{Burning Numbers of $C(n;1,2)$ and $C(n;1,3)$}

We now use the lower bounds presented in Theorem \ref{cnmLB} to find the burning numbers of circulant graphs $C(n;1,2)$ and $C(n;1,3)$.

\begin{thm}\label{circ12}
For any $n\ge5$,
 $b(C(n;1,2))= \left \lceil\frac{1+\sqrt{1+8n}}{4}\right \rceil$.
\end{thm}
\begin{proof}

Let $G \cong C(n;1,2)$, where $n \ge 5$, and let $k=\left \lceil\frac{1+\sqrt{1+8n}}{4}\right \rceil$.
By Theorem \ref{cnmLB}, $b(C(n;1,2) \ge  \left \lceil\frac{1+\sqrt{1+8n}}{4}\right \rceil$.  It therefore suffices to show that there is sequence of vertices $(x_1, x_2, \ldots , x_k)$ that serves as a burning sequence for $G$.

We now define $x_{k-i}$ for each $i \in \{0, \ldots k-1\}$.  First, when $0 \le i \le {k-1}$, we let $x_{k-i}=2i^2+i$.  We note that for any vertex $x$ in $G$, $N_i [x]=[x-2i,x+2i]$. Since $x_{k-i} =2i^2+i$ for $0\le i \le k-2$, it follows that $N_i [x_{k-i} ]=[2i^2-i,2i^2+3i]$ for $0\le i \le k-2$. 

We claim that for all $m=0,2,\dots,k-1$, 
$$[0,2m^2 +3m]\subseteq N_{0} [x_k]\cup N_1 [x_{k-1}]\cup \cdots \cup N_{m} [x_{k-m}]$$

We now proceed by induction.
Note that $N_0 [x_k]=N_0[0]=\{0\}=[0,2(0)^2+3(0)]$, and assume $N_{0} [x_k]\cup N_1 [x_{k-1}]\cup \cdots \cup N_{m} [x_{k-m}]\supseteq [0,2m^2 +3m]$ for some $m$ such that $0\le m\le k-2$. 

Since  $N_{m+1}[x_{k-(m+1)}]=[2(m+1)^{2}-(m+1),2(m+1)^{2}+3(m+1)]=[2m^{2}+3m+1,2(m+1)^{2}+3(m+1)]$,
it follows that $N_{0}[x_{k}]\cup N_{1}[x_{k-1}]\cup\cdots\cup N_{m}[x_{k-m}]\cup N_{m+1}[x_{k-(m+1)}]  \supseteq [0,2(m+1)^{2}+3(m+1)].$  

Therefore, $N_0[x_k]\cup N_1[x_{k-1}]\cup \dots \cup N_{k-1}[x_1]\supseteq [0,2(k-1)^2+3(k-1)]=[0,(2k+1)(k-1)]$.
 
 Since $ k \ge \frac{1+\sqrt{1+8n}}{4}$,
  \begin{eqnarray*}  
  &&(2k+1)(k-1)\\
 &\ge& \left ( \frac{1+\sqrt{1+8n}}{2} +1 \right)  \left ( \frac{1+\sqrt{1+8n}}{4} -1 \right)\\
 &=&n-1
\end{eqnarray*}
Therefore, $N_{k-1}[x_1]\supseteq [2k^2-5k+3,n-1]$, and $N_0[x_k]\cup N_1[x_{k-1}]\cup \dots \cup N_{k-1}[x_1] =
V(G). $

Since  $b(C(n;1,2))\ge \left \lceil\frac{1+\sqrt{1+8n}}{4}\right \rceil$, it follows that $(x_1, x_2, \ldots , x_k)$ is an optimal burning sequence and  $b(C(n;1,2))=\left \lceil\frac{1+\sqrt{1+8n}}{4}\right \rceil$.
\end{proof}

\begin{thm}\label{circ13}
For any $n\ge 7$,   $b(C(n;1,3)) = \left \lfloor \frac{2+\sqrt{3n-2}}{3} \right \rfloor +1.$  \end{thm}

\begin{proof}
Let $G \equiv C(n;1,3)$, where $n \ge 7$. We begin by proving that $b(C(n;1,3)) \ge \left \lfloor \frac{2+\sqrt{3n-2}}{3} \right \rfloor +1.$   Using the first lower bound  given in Theorem \ref{cnmLB}, we obtain $b(C(n;1,m)) \ge 3$ when $7 \le n \le 11$.  When $n \ge 12$, we use the odd case for the second lower bound given in Theorem \ref{cnmLB}.  This gives $b(C(n;1,3)) \ge \left \lceil \frac{2+\sqrt{3n-2}}{3} \right \rceil$.  This generalizes to  $b(C(n;1,3)) \ge \left \lceil \frac{2+\sqrt{3n-2}}{3} \right \rceil$ for all $n \ge 7$.  We note that whenever $\frac{2+\sqrt{3n-2}}{3}$ is non-integer, $\left \lceil \frac{2+\sqrt{3n-2}}{3} \right \rceil =\left \lfloor \frac{2+\sqrt{3n-2}}{3} \right \rfloor +1$.  Therefore, it suffices to show that  $b(C(n;1,3)) >  \frac{2+\sqrt{3n-2}}{3} $ for all $n \ge 7$.

Suppose this is not the case.  Suppose  $b(C(n;1,3)) = \frac{2+\sqrt{3n-2}}{3} $ for some $n \ge 7$.  It follows from the proof of Theorem \ref{cnmLB} that there must be an optimal burning sequence   $(x_{1},x_{2,}.....x_{k})$ such that the elements in $\{N_{k-i}[x_i] \, | \, 1 \le i \le k\}$ are  pair-wise disjoint.  

Without loss of generality, assume that $x_k=0$.  It follows that for some $x$ in the optimal burning sequence, where $x=x_{k - \ell}$ and $\ell \ge 1$, $1 \in N_\ell[x]$, but $0 \not \in N_\ell[x]$.  By Lemma \ref{neighbourhood1}, $N_\ell[x] = \{x-3\ell, x+3\ell\} \cup [x-3\ell+2, x+3\ell-2]$.  Hence, either $1 = x - 3\ell$, $1 = x+3 \ell$, or $1=x-3\ell +2$.  In other words, $x = 1 + 3\ell$, $x = 1-3 \ell$, or $x=-1+3\ell $.  Without loss of generality, we consider the two cases $x = 1 + 4\ell$ and $x=1-3\ell$.

{\it Case 1:} $x = 1 + 3\ell$.  Then $N_\ell[x] = \{1,1+6\ell\} \cup [3,6\ell -1]$.  It follows that $2 \not \in N_\ell[s]$.  Otherwise, we would have $1+6\ell = 2$ and $6\ell -1 =0$, which contradicts the fact that $0 \not \in N_\ell[x]$.  It follows that there must be some vertex $y$ in the optimal burning sequence such that $y = x_{k-\ell'}$ and  $2 \in N_{\ell '}[y]$.  It follows that $0,1 \not \in N_{\ell '}[y]$.  
As with $x$, this gives us three possible values for $y$: $y=2+3 \ell '$, $y = 2-3\ell'$, and $y = 3\ell'$.  Since $0 \not \in N_{\ell '}[y]$, it follows that $y = 2+3 \ell '$.  However, this means $N_{\ell '}[y] = \{2\} \cup [4,6 \ell]$.  This implies $4$ is in both $N_{\ell }[x]$ and $N_{\ell '}[y]$, which is a contradiction.
 
{\it Case 2:}  $x = 1 - 3\ell$.  Then $N_\ell[x] = \{1-6\ell,1 \} \cup [3-6\ell, -1]$.  We again consider a vertex $y$ such that $y = x_{k-\ell'}$ and  $2 \in N_{\ell '}[y]$.  Of the three possible values for $y$ only $y=2+3 \ell '$ gives   $N_{\ell '}[y] \cap N_\ell[x] = \emptyset$.  This means $N_{\ell '}[y] = \{2, 2+6\ell'\} \cup [4, 6 \ell ']$.  It follows that $3$ is in  neither $N_\ell[x]$ nor $N_{\ell'}[y]$, and there must be some vertex $z$ and  positive integer $\ell ''$ such that $3 \in N_{\ell''}[z]$.  However, it is impossible to choose such a vertex $z$ and integer $\ell''$ such that $N_{\ell''}[z]$ is disjoint from both $N_\ell[x]$ and $N_{\ell'}[y]$.  

It follows that there is no optimal burning sequence that results in a set of disjoint neighbourhoods.  Therefore, $b(C(n;1,3)) >  \frac{2+\sqrt{3n-2}}{3}$, and  $b(C(n;1,3)) \ge \left \lfloor \frac{2+\sqrt{3n-2}}{3} \right \rfloor +1$.

\medskip

Next, we assume $k=\left \lfloor \frac{2+\sqrt{3n-2}}{3} \right \rfloor+1$. Then for $0\le i \le {k-1}$, we choose $x_{k-i}=3i^2-i$. Again, it suffices to show that 
$$N_0 [x_k]\cup N_1 [x_{k-1}]\cup \cdots \cup N_{k-1} [x_1]= [0,n-1]$$

Since $N_i[x_{k-i}]=[x_{k-i}-3i+2,x_{k-i}+3i-2]\cup\{x_{k-i}-3i\}\cup\{x_{k-i}+3i\}$ and $x_{k-i}=3i^2-i$, we have that 
$$N_i[x_{k-i}]=[3i^2-4i+2, 3i^2+2i-2]\cup\{3i^2-4i, 3i^2+2i\}$$
We claim for all $m=1,2,\dots,k-1$  that $$N_{0} [x_k]\cup N_1 [x_{k-1}]\cup \cdots \cup N_{m} [x_{k-m}]= [-1,3m^2 +2m-2]\cup\{3m^2+2m\}.$$

This will be proved via induction on $m$, with the  $m=1$ case verified below.

Since $N_0[x_k] = \{0\}$ and $N_1[x_{k-1}] = [1,3] \cup \{-1,5\}$, it follows that 
\begin{eqnarray*}
N_0 [x_k] \cup N_1 [x_{k-1}]&=& [-1,3]\cup \{5\}\\
&=&  [-1,3(1)^2 +2(1)-2]\cup\{3(1)^2+2(1)\}.
\end{eqnarray*}

Now assume $N_{0} [x_k]\cup N_1 [x_{k-1}]\cup \cdots \cup N_{m} [x_{k-m}]\supseteq [-1,3m^2 +2m-2]\cup\{3m^2+2m\}$ for some $m$ such that $1 \le m \le k-2$.  It follows that 
\begin{eqnarray*}
& & N_{0}[x_{k}]\cup N_{1}[x_{k-1}]\cup\cdots\cup N_{m}[x_{k-m}]\cup N_{m+1}[x_{k-(m+1)}]\\
 & = & [-1, 3m^2 +2m-2]\cup\{3m^2+2m\} \cup[3(m+1)^{2}-4(m+1)+2,  \, 3(m+1)^{2}+2(m+1)-2]
\\&& \cup \, \{3(m+1)^2-4(m+1), \, 3(m+1)^2+2(m+1)\}\\
 & = & [-1, 3m^2 +2m-2]\cup\{3m^2+2m\} \cup[3m^2 +2m+1,  \, 3(m+1)^{2}+2(m+1)-2]
\\&& \cup \, \{3m^2+2m-1, \, 3(m+1)^2+2(m+1)\}\\
 %& = & [-1, 3m^2 +2m-2]\cup\{3m^2+2m\}\cup \{3m^2+2m-1\}\cup\{3(m+1)^2+2(m+1)\}\\ &&
%\cup[3m^2+2m+1, 3(m+1)^{2}+2(m+1)-2] \\
 & = & [-1, 3(m+1)^{2}+2(m+1)-2] \cup \{3(m+1)^2+2(m+1)\}
\end{eqnarray*}

It follows by induction that
\begin{eqnarray*}
\bigcup_{i=0}^{k-1} N_i[x_{k-i}] &=  &[-1,3(k-1)^2 +2(k-1)-2]\cup\{3(k-1)^2+2(k-1)\} \\
&= &[-1, 3k^2-4k-1] \cup \{3k^2-4k+1\} \end{eqnarray*}

Since $k=\left \lfloor \frac{2+\sqrt{3n-2}}{3} \right \rfloor +1$, it follows that $k> \frac{2+\sqrt{3n-2}}{3}$. Therefore,
 \begin{eqnarray*}
 & &3k^2-4k-1\\
 &=&(3k-1)(k-1)-2\\
 &>&(\sqrt{3n-2}+1)\left(\frac{\sqrt{3n-2} -1 }{3}\right) -2\\
 &=&n-3
  \end{eqnarray*}
  
  Therefore, $3k^2-4k-1 \ge n-2$ and $[-1, 3k^2-4k-1]=  [-1,n-2] =[0, n-1]$.  It follows that $\bigcup_{i=0}^{k-1} N_i[x_{k-i}] = [0,n-1]$, and $b(C(n;1,3)) \le\left \lfloor \frac{2+\sqrt{3n-2}}{3} \right \rfloor+1$.  Hence,  $$b(C(n;1,3)) = \left \lfloor \frac{2+\sqrt{3n-2}}{3} \right \rfloor+1.$$  
\end{proof}

\section{Burning Numbers of Circulant Graphs of Higher Degree}

\begin{thm}\label{interval}
For any $m \ge 2$ and $n >2m$, $$b(C(n;1,2, \ldots , m)) = \left \lceil \frac{(m-1) + \sqrt{4mn + (m-1)^2}}{2m}\right \rceil.$$
\end{thm}

\begin{proof}
Let $G = b(C(n;1,2, \ldots , m)) $.  We note that for any $x \in V(G)$ and $\ell \le 0$ , $N_{\ell}[x] = [x-\ell m, x+ \ell m]$.  Furthermore,  $ |N_\ell [x]| = 2\ell +1$ for any $\ell \le \frac{n-1}{2}$; otherwise, $ |N_\ell [x]| = n$.  Using an argument similar to that in the proof of Theorem \ref{circ12}, we have 
$b(C(n;1,2, \ldots , m)) \ge \left \lceil \frac{(m-1) + \sqrt{4mn + (m-1)^2}}{2m}\right \rceil.$

Next, we select the burning sequence $(x_1, x_2, \ldots , x_k)$ where $k =   \left \lceil \frac{(m-1) + \sqrt{4mn + (m-1)^2}}{2m}\right \rceil$ and $x_{k-i} = i^2 m +i$ for each $i=0, \ldots , k-1$.  Again, using a proof similar to that for Theorem \ref{circ12}, it can be shown that this sequence burns all the vertices of $G$.  The result follows.
\end{proof}

The {\it lexicographic} product (or {\it wreath} product) of graphs $G$ and $H$ is denoted $G \cdot H$,  where  $V(G\cdot H) = V(G) \times V(H)$, and  $E(G \cdot H ) = \{(x_1, y_1) (x_2, y_2)| x_1x_2 \in E(G) ,   \mbox{ or }  x_1= x_2 \mbox{ and }  y_1y_2 \in E(H)\}$.   In other words, $G \cdot H$ is obtained by (1) replacing each vertex of $G$ with a copy of $H$, and (2) adding all possible edges between two copies of $H$,  whenever they have replaced adjacent vertices of $G$. 
It is well known that the lexicographic product of two circulant graphs is itself a circulant graph.   Specifically, if $G = C(n_1;S)$ and $H = C(n_2, T)$, then $G \cdot  H \cong C(n_1n_2; n_1T \cup  \left( \bigcup_{ s \in S} (n_1 \mathbb Z_{n_2} + s) \right) $ \cite{Br}.  

Lower and upper bounds of the burning number of a lexicographic product $G \cdot H$ were given by Roshanbin \cite{BGM}, as stated the following theorem.

\begin{thm}\label{wreathburnBnds}\cite{BGM}
For any two graphs $G$ and $H$,  we have $$ b(G)\le b(G\cdot H)\le b(G)+2.$$ 
\end{thm}

Therefore, from Theorems, \ref{half},  \ref{circ12}, \ref{circ13} , \ref{interval} and  \ref{wreathburnBnds}, we obtain the following.

\begin{cor}  For any graph $H$, 
\begin{enumerate}
\item $ \left \lceil\frac{1+\sqrt{2n+1}}{2} \right \rceil  \le b\left (C\left(n;1,\frac n2 \right)\cdot  H \right ) \le \left \lceil\frac{1+\sqrt{2n+1}}{2} \right \rceil  + 2. $
\item $\left \lceil\frac{1+\sqrt{1+8n}}{4} \right \rceil \le b(C(n;1,2)\cdot H ) \le \left \lceil\frac{1+\sqrt{1+8n}}{4} \right \rceil + 2. $
\item $ \left \lfloor \frac{2+\sqrt{3n-2}}{3} \right \rfloor +1 \le b( C(n;1,3)\cdot H ) \le \left \lfloor \frac{2+\sqrt{3n-2}}{3} \right \rfloor +3.$
\item $\left \lceil \frac{(m-1) + \sqrt{4mn + (m-1)^2}}{2m}\right \rceil \le  b(C(n;1,2, \ldots , m) \cdot H) \le  \left \lceil \frac{(m-1) + \sqrt{4mn + (m-1)^2}}{2m}\right \rceil +2$.
\end{enumerate}
\end{cor}

Finally from Theorem  \ref{order}, we have the following, more general result.

\begin{cor}
For any graph $H$,  $b(C(n;1,m) \cdot H)$ is on the order of $\sqrt\frac nm$.
\end{cor}

While these results hold for any graph $H$, when  $H = C(n'; T)$ for some $n' \ge 4$ and an appropriate choice of set $T$, they establish bounds on the burning numbers of  infinite families of circulant graphs.

\end{document}